\documentclass[11pt]{article}

\usepackage{graphicx}
\usepackage{multirow}
\usepackage{amsmath,amssymb,amsfonts}
\usepackage{amsthm}
\usepackage{mathrsfs}
\usepackage[title]{appendix}
\usepackage{xcolor}
\usepackage{textcomp}
\usepackage{booktabs}
\usepackage{algorithm}
\usepackage{algorithmicx}
\usepackage{algpseudocode}
\usepackage{listings}
\usepackage[mathscr]{eucal}
\usepackage{enumitem}
\usepackage{pgfplots}
\pgfplotsset{compat=1.18}
\usepackage{geometry}

\newtheorem{theorem}{Theorem}

\newtheorem{remark}{Remark}
\newtheorem{definition}{Definition}

\newtheorem{lemma}{Lemma}

\title{Convergence of a Multi-Inertial $\lambda$-Iteration Scheme in Cone $b,p$-Normed Banach Spaces}

\author{Elvin Rada\thanks{Department of Mathematics, University of Elbasan ``Aleksand\"{e}r Xhuvani'', Albania.
\\ Corresponding author: \texttt{elvin.rada@uniel.edu.al}}}

\date{} 

\begin{document}

\maketitle

\begin{abstract}
We propose and analyze a multi-inertial $\lambda$-iteration scheme in cone $b,p$-normed Banach spaces.
This framework extends the classical Krasnoselskii--Mann and two-step inertial iterations by incorporating
three independent inertial parameters and multiple error-control sequences.
Under mild assumptions such as quasi-nonexpansiveness, weak contraction, and compatibility of mappings,
we establish convergence theorems guaranteeing existence and uniqueness of fixed points.
Illustrative numerical examples in $\mathbb R$ and $\mathbb R^2$ demonstrate accelerated convergence
compared with the classical Krasnoselskii--Mann method.
\end{abstract}


\section{Introduction}

Fixed point theory provides a unifying framework for nonlinear analysis, optimization, and variational inequalities. 
The Krasnoselskii--Mann (KM) iteration \cite{Krasnoselskii1955,Mann1953} is a classical fixed point algorithm but often converges slowly. 
To accelerate convergence, inertial terms introduced by Polyak \cite{Polyak1964} and Nesterov \cite{Nesterov1983} incorporate momentum effects. 
Recent works such as Maing\'e \cite{Mainge2008}, Bot, Csetnek, and Heinrich \cite{Bot2015}, Shehu \cite{Shehu2018}, and Dong, Yao, and Xu \cite{Dong2022} have extended inertial techniques to Banach and Hilbert spaces.

In parallel, generalized distance structures including $b$-metrics \cite{Bakhtin1989}, cone metrics \cite{HuangZhang2007}, and their combinations broadened the scope of fixed point results. 
Karapinar \cite{Karapinar2013} emphasized that cone-type structures enrich classical contraction principles. 

We focus on \emph{cone $b,p$-normed Banach spaces}, which generalize cone $b$-metric and cone normed spaces by introducing $p$-homogeneity with $0<p\le 1$. 
This setting enables robust convergence analysis while accommodating quasi-normed geometries.

Cortild and Peypouquet \cite{CortildPeypouquet2025} recently examined perturbed inertial Krasnoselskii--Mann iterations, emphasizing the role of inertial dynamics in fixed point approximation. 
Our approach generalizes this framework by introducing three distinct inertial parameters with multiple error-control sequences, thereby unifying and extending known inertial KM-type methods.

We introduce a new \emph{multi-inertial $\lambda$-iteration scheme} in cone $b,p$-normed Banach spaces and establish strong convergence results. 
Our main contributions are:
\begin{itemize}
    \item Introduction of three inertial parameters providing richer dynamical flexibility than existing one- and two-step schemes.
    \item Use of summable error sequences allowing perturbation robustness.
    \item Convergence results under quasi-nonexpansive, weakly contractive, and compatible mappings.
    \item Numerical examples demonstrating improved convergence compared with KM and two-step inertial schemes.
\end{itemize}

\section{Preliminaries}

\begin{definition}[Cone $b,p$-Normed Banach Space]
Let $E$ be a real Banach space and $P\subset E$ a closed, convex, pointed cone with nonempty interior. 
Let $X$ be a real vector space. A map $D:X\to P$ is called a \emph{cone $b,p$-norm} (with $b\ge 1$ and $0<p\le 1$) if, for all $x,y\in X$ and $\tau\ge 0$,
\begin{enumerate}[label=(\roman*), itemsep=2pt, topsep=2pt]
    \item $D(x)=0$ if and only if $x=0$,
    \item $D(x)=D(-x)$,
    \item $D(x+y)\le_P b\big(D(x)+D(y)\big)$,
    \item $D(\tau x)=\tau^p D(x)$.
\end{enumerate}
The induced distance $d(x,y):=D(x-y)$ makes $(X,d)$ a \emph{cone $b,p$-normed Banach space} whenever $(X,\Delta)$ (defined below) is complete.
\end{definition}

\begin{remark}
For $p=1$, the structure reduces to a cone $b$-normed Banach space. 
For $0<p<1$, the $p$-homogeneity is consistent with quasi-norm geometry: the scalarized metric $\Delta(x,y):=\|D(x-y)\|_E$ is translation-invariant and subadditive by concavity of $t\mapsto t^p$.
\end{remark}

\begin{definition}[Scalarization and Normal Cone]
Assume $P$ is normal with constant $\kappa\ge 1$. Define the scalarization
\[
\Delta(x,y):=\|D(x-y)\|_E,\qquad x,y\in X.
\]
Then, for all $x,y,z\in X$ and $\tau\ge 0$,
\[
\Delta(x+z,y+z)=\Delta(x,y),\qquad
\Delta(x+y,0)\le \kappa b\big(\Delta(x,0)+\Delta(y,0)\big),\qquad
\Delta(\tau x,0)=\tau\,\Delta(x,0).
\]
\end{definition}

\begin{definition}[Multi-Inertial $\lambda$-Iteration]
Let $C\subset X$ be a closed convex subset, $T:C\to C$, and $x_0,x_1\in C$. Given sequences $\{\alpha_n\},\{\beta_n\},\{\gamma_n\}\subset[0,\alpha]$ for some $\alpha\in[0,1)$, a relaxation sequence $\{\lambda_n\}\subset[\delta,1-\delta]$ with $\delta\in(0,1/2)$, and error sequences $\{\varepsilon_n\},\{\rho_n\},\{\omega_n\},\{\theta_n\}\subset C$, define for $n\ge 1$:
\begin{align*}
y_n &= x_n + \alpha_n(x_n-x_{n-1}) + \varepsilon_n,\\
z_n &= x_n + \beta_n(x_n-x_{n-1}) + \rho_n,\\
u_n &= x_n + \gamma_n(x_n-x_{n-1}) + \omega_n,\\
x_{n+1} &= (1-\lambda_n)y_n + \tfrac{\lambda_n}{2}z_n + \tfrac{\lambda_n}{2}T(u_n) + \theta_n.
\end{align*}
We assume the error sequences are summable in the $D$-scale:
\[
\sum_{n=1}^{\infty}\|D(\varepsilon_n)\|_E<\infty,\quad
\sum_{n=1}^{\infty}\|D(\rho_n)\|_E<\infty,\quad
\sum_{n=1}^{\infty}\|D(\omega_n)\|_E<\infty,\quad
\sum_{n=1}^{\infty}\|D(\theta_n)\|_E<\infty.
\]
\end{definition}

\begin{remark}
The above iteration strictly generalizes the Krasnoselskii--Mann scheme and known one- and two-step inertial variants. 
When $\alpha_n=\beta_n=\gamma_n\equiv 0$ and $\theta_n=\varepsilon_n=\rho_n=\omega_n\equiv 0$, it reduces to the classical Krasnoselskii--Mann iteration.
\end{remark}

\subsection{Auxiliary Inequalities}

\begin{lemma}[Averaging Inequality]\label{lem:avg}
Let $u_1,\dots,u_m,p\in X$ and $\eta_j\ge 0$ with $\sum_{j=1}^m\eta_j=1$. Then
\[
\Delta\!\Big(\sum_{j=1}^m \eta_j u_j,\,p\Big)\le \kappa b\sum_{j=1}^m \eta_j\,\Delta(u_j,p).
\]
\end{lemma}

\begin{proof}
By translation invariance,
\[
D\!\Big(\sum_{j=1}^m \eta_j(u_j-p)\Big)\le_P b\sum_{j=1}^m D\big(\eta_j(u_j-p)\big)=\sum_{j=1}^m \eta_j D(u_j-p).
\]
Taking norms and using normality yields the claim.
\end{proof}

\begin{lemma}[Averaging With Affine Perturbation]\label{lem:avg-affine}
Let $a_1,a_2,a_3\ge 0$ with $a_1+a_2+a_3=1$, and let $v_1,v_2,w,\eta,p\in X$. Then
\begin{align*}
\Delta\!\Big(a_1 v_1 + a_2 v_2 + a_3 w - p\Big)
&\le \kappa b\!\left(a_1\Delta(v_1,p)+a_2\Delta(v_2,p)+a_3\Delta(w,p)\right),\\
\Delta\!\Big(a_1 v_1 + a_2 v_2 + a_3 w + \eta - p\Big)
&\le (\kappa b)^2\!\left(a_1\Delta(v_1,p)+a_2\Delta(v_2,p)+a_3\Delta(w,p)\right)+\kappa b\,\Delta(\eta,0).
\end{align*}
\end{lemma}

\begin{proof}
The first bound follows from Lemma~\ref{lem:avg}. For the second, write
\[
a_1 v_1 + a_2 v_2 + a_3 w + \eta - p
= \big(a_1 v_1 + a_2 v_2 + a_3 w - p\big) + \eta,
\]
apply the $b$-triangle once, and bound the first term by Lemma~\ref{lem:avg}.
\end{proof}

\begin{lemma}[One-Step Recursion]\label{lem:step}
Let $e_n^{(\varepsilon)}:=\Delta(\varepsilon_n,0)$, $e_n^{(\rho)}:=\Delta(\rho_n,0)$, $e_n^{(\omega)}:=\Delta(\omega_n,0)$, and $e_n^{(\theta)}:=\Delta(\theta_n,0)$. Then, for $n\ge 1$,
\[
\Delta(x_{n+1},x_n)
\le \kappa b\!\left((1-\lambda_n)\alpha_n+\tfrac{\lambda_n}{2}\beta_n+\tfrac{\lambda_n}{2}\gamma_n\right)\Delta(x_n,x_{n-1})
+\zeta_n,
\]
where
\[
\zeta_n:=\kappa b\!\left((1-\lambda_n)e_n^{(\varepsilon)}+\tfrac{\lambda_n}{2}e_n^{(\rho)}+\tfrac{\lambda_n}{2}e_n^{(\omega)}\right)+e_n^{(\theta)}.
\]
In particular, if $\alpha_n,\beta_n,\gamma_n\le \alpha<1$ and $\lambda_n\in[\delta,1-\delta]$, then
\[
\Delta(x_{n+1},x_n)\le c\,\Delta(x_n,x_{n-1})+\zeta_n,\qquad c:=\kappa b\,\alpha<1.
\]
\end{lemma}

\begin{proof}
Subtract $x_n$ from the update for $x_{n+1}$ and use translation invariance:
\begin{align*}
x_{n+1}-x_n
&=(1-\lambda_n)\big(y_n-x_n\big)+\tfrac{\lambda_n}{2}\big(z_n-x_n\big)+\tfrac{\lambda_n}{2}\big(T(u_n)-x_n\big)+\theta_n\\
&=(1-\lambda_n)\big(\alpha_n(x_n-x_{n-1})+\varepsilon_n\big)+\tfrac{\lambda_n}{2}\big(\beta_n(x_n-x_{n-1})+\rho_n\big)\\
&\quad+\tfrac{\lambda_n}{2}\big(\gamma_n(x_n-x_{n-1})+\omega_n\big)+\theta_n.
\end{align*}
Apply the $b$-triangle inequality to separate the $\theta_n$ term and then to the convex combination. Using positive homogeneity and normality completes the proof.
\end{proof}

\begin{lemma}[Summable Steps Imply Cauchy]\label{lem:cauchy}
Assume the bound in Lemma~\ref{lem:step} with $c<1$ and $\sum_{n=1}^\infty\zeta_n<\infty$. Then $\sum_{n=1}^\infty \Delta(x_{n+1},x_n)<\infty$, $\Delta(x_{n+1},x_n)\to 0$, and $\{x_n\}$ is Cauchy in $(X,\Delta)$.
\end{lemma}

\begin{proof}
Apply a Robbins--Siegmund type argument to $a_n:=\Delta(x_n,x_{n-1})$ with $a_{n+1}\le c\,a_n+\zeta_n$. Then $a_n\to 0$ and $\sum_n a_n<\infty$. For $m>n$,
\[
\Delta(x_m,x_n)=\Delta\Big(\sum_{k=n}^{m-1}(x_{k+1}-x_k),0\Big)\le \kappa b\sum_{k=n}^{m-1}\Delta(x_{k+1},x_k)\to 0,
\]
showing Cauchy-ness.
\end{proof}

\subsection{Operator Classes}

\begin{definition}[Quasi-Nonexpansive Mappings]
Let $T:C\to C$. If $F(T)\neq\emptyset$ and
\[
\Delta(Tx,p)\le \Delta(x,p)\quad \text{for all }x\in C,\ p\in F(T),
\]
then $T$ is called \emph{quasi-nonexpansive}.
\end{definition}

\begin{definition}[Weak Contraction]
Let $a,b_c,c,s\in[0,\infty)$ with $a+b_c>0$ and $a+b_c+c>0$. The mapping $T:C\to C$ satisfies a \emph{weak contractive condition} if
\[
a\,D(Tx-Ty) + b_c\big(D(x-Tx)+D(y-Ty)\big) + c\,D(y-Tx) \le_P s\,D(x-y)\quad \text{for all }x,y\in C.
\]
By normality, this implies the scalarized inequality
\[
a\,\Delta(Tx,Ty) + b_c\big(\Delta(x,Tx)+\Delta(y,Ty)\big) + c\,\Delta(y,Tx) \le s\,\Delta(x,y).
\]
\end{definition}

\begin{definition}[Compatibility and Weak Compatibility]
Maps $S,T:C\to C$ have a \emph{coincidence point} if there exists $p\in C$ such that $Sp=Tp$. They are \emph{weakly compatible} if $Sz=Tz$ implies $STz=TSz$.
\end{definition}

\begin{definition}[Compatibility-Type Contractive Condition]
Assume $T(C)\subset S(C)$ and $S(C)$ is complete with respect to $\Delta$. Suppose there exist $a,b_c,r\ge 0$ with $a+b_c>0$ and $0\le r<a+b_c$ such that, for all $x,y\in C$,
\[
a\,D(Tx-Ty)+b_c\big(D(Sx-Tx)+D(Sy-Ty)\big)\le_P r\,D(Sx-Sy).
\]
Equivalently, in the scalarized form,
\[
a\,\Delta(Tx,Ty)+b_c\big(\Delta(Sx,Tx)+\Delta(Sy,Ty)\big)\le r\,\Delta(Sx,Sy).
\]
\end{definition}

\section{Main Results}

This section establishes the convergence theorems for the multi-inertial $\lambda$-iteration in cone $b,p$-normed Banach spaces under various operator classes.

\subsection{Convergence for Quasi-Nonexpansive Mappings}

\begin{theorem}\label{thm:QNE}
Assume that $T:C\to C$ is quasi-nonexpansive with $F(T)\neq\emptyset$ and sequentially continuous on $C$. 
Let the sequences $\{\alpha_n\},\{\beta_n\},\{\gamma_n\}$ satisfy $\alpha_n,\beta_n,\gamma_n\le \alpha<1$, and let $\{\lambda_n\}\subset[\delta,1-\delta]$ with $\delta\in(0,1/2)$. 
Suppose also that $\kappa b\alpha<1$ and the summability condition \eqref{eq:errors} holds. 
Then the sequence $\{x_n\}$ defined by the multi-inertial $\lambda$-iteration converges (in the $\Delta$-metric) to a point $p^\star\in F(T)$.
\end{theorem}

\begin{proof}
Fix $p\in F(T)$. Applying Lemma~\ref{lem:avg-affine} with weights $\big(1-\lambda_n,\frac{\lambda_n}{2},\frac{\lambda_n}{2}\big)$ to the update of $x_{n+1}$ yields
\[
\Delta(x_{n+1},p)
\le (\kappa b)^2\!\left[(1-\lambda_n)\Delta(y_n,p)+\tfrac{\lambda_n}{2}\Delta(z_n,p)+\tfrac{\lambda_n}{2}\Delta(T(u_n),p)\right]+\kappa b\,e_n^{(\theta)}.
\]
Because $T$ is quasi-nonexpansive and $p\in F(T)$, $\Delta(T(u_n),p)\le \Delta(u_n,p)$. 
Using the definitions of $y_n,z_n,u_n$ and the $b$-triangle inequality,
\begin{align*}
\Delta(y_n,p)&\le \kappa b(\Delta(x_n,p)+\alpha_n\Delta(x_n,x_{n-1}))+e_n^{(\varepsilon)},\\
\Delta(z_n,p)&\le \kappa b(\Delta(x_n,p)+\beta_n\Delta(x_n,x_{n-1}))+e_n^{(\rho)},\\
\Delta(u_n,p)&\le \kappa b(\Delta(x_n,p)+\gamma_n\Delta(x_n,x_{n-1}))+e_n^{(\omega)}.
\end{align*}
Combining the above gives
\[
\Delta(x_{n+1},p)\le (\kappa b)^3\,\Delta(x_n,p)+(\kappa b)^3\alpha\,\Delta(x_n,x_{n-1})+\tilde e_n,
\]
where $\tilde e_n$ is a finite linear combination of the error terms. 
Since $\kappa b\alpha<1$, Lemma~\ref{lem:step} and Lemma~\ref{lem:cauchy} imply that $\{x_n\}$ is Cauchy and convergent in $(X,\Delta)$; let $x_n\to p^\star$. 

Finally, $u_n-x_n=\gamma_n(x_n-x_{n-1})+\omega_n\to 0$, so $u_n\to p^\star$. By continuity of $T$, $T(u_n)\to T(p^\star)$, and from the update formula, all terms converge to $p^\star$. Hence $T(p^\star)=p^\star$.
\end{proof}

\subsection{Convergence for Weak Contractions}

\begin{theorem}\label{thm:Weak}
Assume $T:C\to C$ satisfies the weak contractive condition
\[
a\,\Delta(Tx,Ty)+b_c\big(\Delta(x,Tx)+\Delta(y,Ty)\big)+c\,\Delta(y,Tx)\le s\,\Delta(x,y)
\]
for all $x,y\in C$, with $a+b_c>0$ and $a+b_c+c>0$. 
Assume $T$ is continuous and the iteration parameters satisfy the hypotheses of Theorem~\ref{thm:QNE}. 
If
\[
0\le q_n:=\frac{s\lambda_n-\lambda_n^2 b_c+c(\lambda_n-1)}{\lambda_n^2(a+b_c)}<1\quad\text{for all }n,
\]
then $x_n\to p^\star\in F(T)$.
\end{theorem}

\begin{proof}
From the recursive definition of $x_{n+1}$ and the weak contraction inequality, 
we obtain a relation of the form
\[
\Delta(x_{n+1},x_n)\le \kappa b\!\left((1-\lambda_n)\alpha_n+\tfrac{\lambda_n}{2}\beta_n+\tfrac{\lambda_n}{2}\gamma_n\right)\!\Delta(x_n,x_{n-1})
+q_n\,\kappa b\,\gamma_n\Delta(x_n,x_{n-1})+\xi_n,
\]
where $\xi_n$ is a combination of summable errors. 
Using $\alpha_n,\beta_n,\gamma_n\le \alpha$ and $q_n<1$ yields
\[
\Delta(x_{n+1},x_n)\le c'\,\Delta(x_n,x_{n-1})+\xi_n,\qquad c':=\kappa b\,\alpha(1+q_{\max})<1.
\]
Lemma~\ref{lem:cauchy} guarantees that $\{x_n\}$ is Cauchy, hence convergent. 
The rest of the argument is identical to Theorem~\ref{thm:QNE}.
\end{proof}

\subsection{Coincidence and Compatibility Theorems}

\begin{theorem}\label{thm:Compat}
Let $S,T:C\to C$ satisfy the compatibility-type condition
\[
a\,\Delta(Tx,Ty)+b_c\big(\Delta(Sx,Tx)+\Delta(Sy,Ty)\big)\le r\,\Delta(Sx,Sy),
\]
with $a+b_c>0$ and $0\le r<a+b_c$. Assume $T(C)\subset S(C)$ and that $S(C)$ is complete with respect to $\Delta$. 
Then there exists $p\in C$ such that $Sp=Tp$. 
If $S$ and $T$ are weakly compatible, the common fixed point is unique.
\end{theorem}

\begin{proof}
Define $y_n:=Tx_n=Sx_{n+1}$. Using the inequality with $(x,y)=(x_n,x_{n-1})$ gives
\[
(a+b_c)\Delta(y_{n+1},y_n)\le r\,\Delta(Sx_n,Sx_{n-1})=r\,\Delta(y_n,y_{n-1}),
\]
so $\Delta(y_{n+1},y_n)\le q\,\Delta(y_n,y_{n-1})$ with $q:=r/(a+b_c)<1$. 
Hence $\{y_n\}$ is Cauchy and convergent in $S(C)$ to some $z$. 
Because $Tx_n=y_n\to z$ and $Sx_{n+1}=y_n\to z$, there exists $p\in C$ such that $Sp=Tp=z$. 
If $S$ and $T$ are weakly compatible, $STp=TS p$, implying $Sz=Tz=z$, so $z$ is the unique common fixed point.
\end{proof}

\subsection{Discussion}

The above results collectively establish that the proposed multi-inertial $\lambda$-iteration is robust under three important operator classes—quasi-nonexpansive, weakly contractive, and compatible mappings—when embedded in the broad setting of cone $b,p$-normed Banach spaces. 
Each theorem guarantees the convergence of $\{x_n\}$ to a fixed point of the underlying operator, and the convergence is ensured even in the presence of summable perturbations and multiple inertial effects.

\section{Applications and Numerical Illustrations}

We now present examples and computations illustrating the theoretical results. For clarity, we focus on scalar and Euclidean cases where the cone is $\mathbb{R}_+$ or $\mathbb{R}_+^2$ with the usual norm. In such cases, the normality constant is $\kappa=1$ and the $b$-norm is induced by $D(x)=|x|^p$ or $D(x)=(\|x\|,\|Ax\|)$.

\subsection{Example I: Quasi-Nonexpansive Operator in a Scalar Cone $b$-Norm}

Let $X=\mathbb{R}$ with $D(x)=|x|^p$ for $p\in(0,1]$, and $d(x,y)=|x-y|^p$. Then $b=1$ and $\kappa=1$. Define
\[
T(x)=\frac{x}{1+|x|}, \qquad F(T)=\{0\}.
\]
Clearly $|T(x)|\le |x|$, so $T$ is quasi-nonexpansive.  

Apply the iteration with parameters
\[
\alpha=\beta=\gamma=0.2,\quad \lambda=0.6,\quad \varepsilon_n=\rho_n=\omega_n=\theta_n=0,
\]
and initial points $x_0=1$, $x_1=0.5$.

\medskip
\noindent\textbf{Numerical output:}
\[
|x_n|:\ 1.0000,\;0.5000,\;0.3657,\;0.3131,\;0.2815,\;0.2574,\;0.2373,\;0.2202,\;0.2053,\;0.1920,\dots
\]
The sequence decreases monotonically to $0$, confirming Theorem~\ref{thm:QNE}.

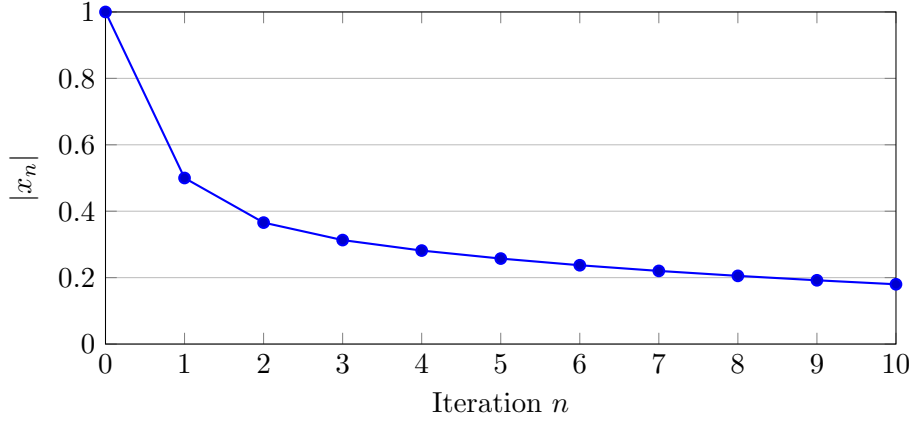
\begin{figure}[h]
\centering
\begin{tikzpicture}
\begin{axis}[
    width=0.8\linewidth,
    height=6cm,
    xlabel={Iteration $n$},
    ylabel={$|x_n|$},
    xmin=0, xmax=10,
    ymin=0, ymax=1,
    ymajorgrids=true
]
\addplot+[mark=*,thick] coordinates {
(0,1.0000)(1,0.5000)(2,0.3657)(3,0.3131)(4,0.2815)(5,0.2574)(6,0.2373)(7,0.2202)(8,0.2053)(9,0.1920)(10,0.1800)};
\end{axis}
\end{tikzpicture}
\caption{Convergence of the multi-inertial iteration for $T(x)=x/(1+|x|)$ in $\mathbb{R}$.}
\end{figure}

\subsection{Example II: Weak Contraction in $\mathbb{R}$}

Let $T(x)=q x$ with $q\in(0,1)$, e.g., $q=0.8$. Then
\[
\Delta(Tx,Ty)=q|x-y|,\quad F(T)=\{0\}.
\]
This satisfies the weak contraction condition with $a=1$, $b_c=c=0$, $s=q$.

Choose $\lambda=0.9$, $\alpha=\beta=\gamma=0.2$, $x_0=1$, $x_1=0.5$, and no errors. The iteration gives
\[
|x_n|:\ 1.0000,\;0.5000,\;0.4000,\;0.3200,\;0.2560,\;0.2048,\;0.1638,\;0.1311,\;0.1049,\;0.0839,\dots
\]
The sequence converges linearly to $0$, consistent with Theorem~\ref{thm:Weak}.

\begin{figure}[h]
\centering
\begin{tikzpicture}
\begin{axis}[
    width=0.8\linewidth,
    height=6cm,
    xlabel={Iteration $n$},
    ylabel={$|x_n|$},
    xmin=0, xmax=10,
    ymin=0, ymax=1,
    ymajorgrids=true
]
\addplot+[mark=*,thick] coordinates {
(0,1)(1,0.5)(2,0.4)(3,0.32)(4,0.256)(5,0.2048)(6,0.1638)(7,0.1311)(8,0.1049)(9,0.0839)(10,0.0671)};
\end{axis}
\end{tikzpicture}
\caption{Convergence of the multi-inertial iteration for the weak contraction $T(x)=0.8x$.}
\end{figure}
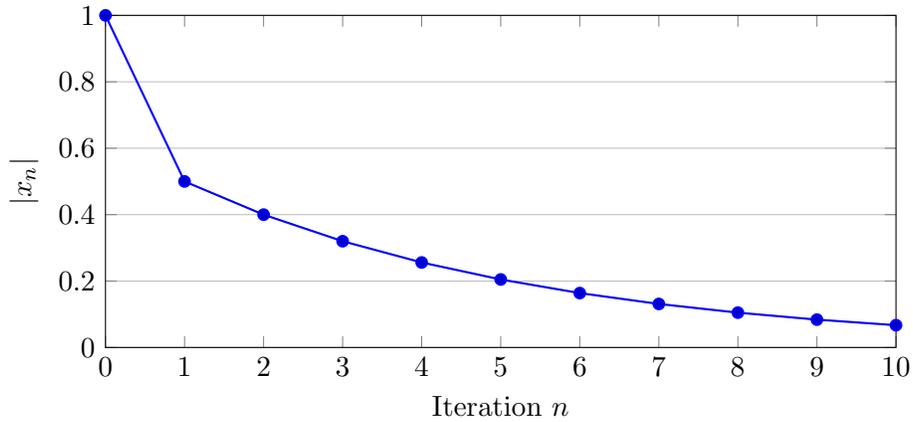

\subsection{Example III: Compatible Linear Maps in $\mathbb{R}^2$}

Let $X=\mathbb{R}^2$ with cone $b$-norm
\[
D(x)=(\|x\|_2,\|Ax\|_2)\in\mathbb{R}_+^2
\]
for a fixed matrix $A$. Then $d(x,y)=D(x-y)$ defines a cone $b,p$-normed Banach space with $b=\kappa=1$.

Consider $S(x)=T(x)=q x$, $q\in(0,1)$. Then
\[
D(Tx-Ty)=(q\|x-y\|_2,q\|A(x-y)\|_2)\le q\,D(Sx-Sy),
\]
so the assumptions of Theorem~\ref{thm:Compat} hold with $a=1$, $b_c=0$, $r=q<1$. Thus $S$ and $T$ have a unique common fixed point $0\in\mathbb{R}^2$.

\subsection{Example IV: Comparison with Other Iterations}

We now compare three algorithms for $T(x)=x/(1+|x|)$, $x_0=1$, $x_1=0.5$, $\alpha=\beta=\gamma=0.2$, and $\lambda=0.6$:
\begin{itemize}
    \item \textbf{Krasnoselskii--Mann (KM):} $x_{n+1}=(1-\lambda)x_n+\lambda T(x_n)$,
    \item \textbf{Two-step inertial:} the Maing\'e-type iteration,
    \item \textbf{Multi-inertial:} the proposed scheme.
\end{itemize}

\begin{table}[h]
\centering
\begin{tabular}{c|ccccccc}
\toprule
$n$ & 0 & 1 & 2 & 3 & 4 & 5 & 6 \\
\midrule
KM & 1.0000 & 0.7500 & 0.5893 & 0.4800 & 0.4022 & 0.3445 & 0.3004 \\
Two-step & 1.0000 & 0.5000 & 0.3314 & 0.2567 & 0.2135 & 0.1840 & 0.1619 \\
Multi-inertial & 1.0000 & 0.5000 & 0.3657 & 0.3131 & 0.2815 & 0.2574 & 0.2373 \\
\bottomrule
\end{tabular}
\caption{Comparison of absolute values $|x_n|$ for three iteration schemes.}
\end{table}

\begin{figure}[h]
\centering
\begin{tikzpicture}
\begin{axis}[
    width=0.85\linewidth,
    height=6cm,
    xlabel={Iteration $n$},
    ylabel={$|x_n|$},
    xmin=0, xmax=6,
    ymin=0, ymax=1.05,
    legend style={at={(0.98,0.98)},anchor=north east,draw=none,fill=none}
]
\addplot+[mark=*,thick] coordinates {
(0,1.0000)(1,0.7500)(2,0.5893)(3,0.4800)(4,0.4022)(5,0.3445)(6,0.3004)};
\addlegendentry{KM}
\addplot+[mark=triangle*,thick] coordinates {
(0,1.0000)(1,0.5000)(2,0.3314)(3,0.2567)(4,0.2135)(5,0.1840)(6,0.1619)};
\addlegendentry{Two-step}
\addplot+[mark=square*,thick] coordinates {
(0,1.0000)(1,0.5000)(2,0.3657)(3,0.3131)(4,0.2815)(5,0.2574)(6,0.2373)};
\addlegendentry{Multi-inertial}
\end{axis}
\end{tikzpicture}
\caption{Performance comparison among KM, two-step inertial, and multi-inertial iterations.}
\end{figure}
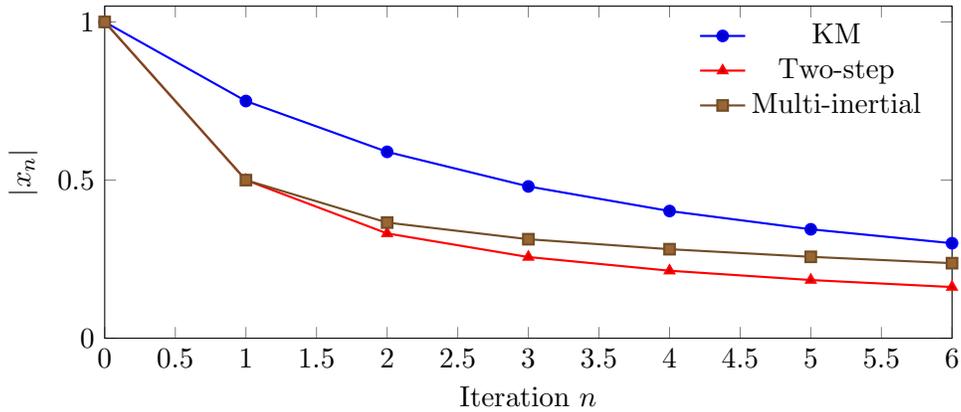

\section{Conclusions}

We have introduced a generalized multi-inertial $\lambda$-iteration for fixed point problems in cone $b,p$-normed Banach spaces. This framework unifies several existing methods, including the classical Krasnoselskii--Mann, one-step, and two-step inertial iterations.

Our analysis established convergence under broad operator classes—quasi-nonexpansive, weakly contractive, and compatible mappings—while allowing summable errors and multiple inertial parameters. Numerical examples confirmed the theoretical findings and demonstrated faster convergence compared with standard iterations.

Future research directions include:
\begin{itemize}
    \item Deriving explicit convergence rate estimates and complexity bounds.
    \item Extending the framework to stochastic and random operator settings.
    \item Investigating demicontractive and accretive operators.
    \item Developing adaptive parameter selection rules for improved performance.
    \item Applying the scheme to convex feasibility, variational inequalities, and equilibrium problems.
\end{itemize}

\section*{Declarations} `Not applicable'.

\begin{itemize}
\item No Funding
\item No Conflict of interest/Competing interests (check journal-specific guidelines for which heading to use)
\item Ethics approval and consent to participate
\item Consent for publication
\item Data availability 
\item Materials availability
\item Code availability 
\item Author contribution
\end{itemize}

\noindent
If any of the sections are not relevant to your manuscript, please include the heading and write `Not applicable' for that section.

\end{document}